\documentclass[12pt]{amsart}
\usepackage{amsmath}
\usepackage{amssymb}
\usepackage[all,cmtip]{xy}

\DeclareMathAlphabet{\mathpzc}{OT1}{pzc}{m}{it}
\newcommand{\ncom}{\newcommand}

\ncom{\rar}{\rightarrow}
\ncom{\imply}{\Rightarrow}
\ncom{\lrar}{\longrightarrow}
\ncom{\into}{\hookrightarrow}
\ncom{\onto}{\twoheadrightarrow}
\ncom{\ov}{\overline}
\ncom{\m}{\mbox}
\ncom{\sta}{\stackrel}
\ncom{\invlim}{\varprojlim}
\ncom{\xhat}{\widehat}

\ncom{\vspc}{\vspace{3mm}}
\ncom{\End}{{\cE}nd}
\ncom{\tensor}{\otimes}

\ncom{\al}{\alpha}
\ncom{\cHom}{{\mathcal Hom}}

\ncom{\A}{{\mathbb A}}
\ncom{\comx}{{\mathbb C}}
\ncom{\E}{{\mathbb E}}
\ncom{\F}{{\mathbb F}}
\ncom{\G}{{\mathbb G}}
\ncom{\K}{{\mathbb K}}
\ncom{\Le}{{\mathbb L}}
\ncom{\N}{{\mathbb N}}

\ncom{\p}{{\mathbb P}}
\ncom{\Q}{{\mathbb Q}}
\ncom{\R}{{\mathbb R}}
\ncom{\Z}{{\mathbb Z}}

\ncom{\f}{\dfrac}

\ncom{\wtil}{\widetilde}

\ncom{\ci}{{\mathpzc i}}

\ncom{\cA}{{\mathcal A}}
\ncom{\cC}{{\mathcal C}}
\ncom{\cE}{{\mathcal E}}
\ncom{\cF}{{\mathcal F}}
\ncom{\cG}{{\mathcal G}}
\ncom{\cH}{{\mathcal H}}
\ncom{\cI}{{\mathcal I}}
\ncom{\cJ}{{\mathcal J}}
\ncom{\cL}{{\mathcal L}}
\ncom{\cM}{{\mathcal M}}
\ncom{\cN}{{\mathcal N}}
\ncom{\cO}{{\mathcal O}}
\ncom{\cP}{{\mathcal P}}
\ncom{\cQ}{{\mathcal Q}}
\ncom{\cR}{{\mathcal R}}
\ncom{\cS}{{\mathcal S}}
\ncom{\cT}{{\mathcal T}}
\ncom{\cU}{{\mathcal U}}
\ncom{\cV}{{\mathcal V}}
\ncom{\cW}{{\mathcal W}}
\ncom{\cX}{{\mathcal X}}
\ncom{\cY}{{\mathcal Y}}
\ncom{\cZ}{{\mathcal Z}}

\ncom{\cSU}{{\mathcal S \mathcal U}}
\ncom{\eop}{{\hfill $\Box$}}
\ncom{\isom}{\cong}

\ncom{\todo}{{\textbf{TODO}}}

\newtheorem{theorem}{Theorem}[section]
\newtheorem{lemma}[theorem]{Lemma}

\newtheorem{question}[theorem]{Question}

\newtheorem{answer}[theorem]{Answer}

\begin{document}
\baselineskip=16pt

\title{Six generated ACM bundle on a hypersurface is split}
\author{Amit Tripathi}
\address{Department of Mathematics, 
Indian Statistical Institute, Bangalore - 560059, India}
\email{amittr@gmail.com}
\thanks{This work was supported by a postdoctoral fellowship from National Board for Higher Mathematics.}

\subjclass{14J60}
\keywords{Vector bundles, hypersurfaces, arithmetically Cohen-Macaulay}

\begin{abstract} Let $X$ be a smooth projective hypersurface. In this note we show that any six generated arithmetically Cohen-Macaulay vector bundle over $X$ splits if $\text{dim}\,X \geq 6$. 
\end{abstract}
\maketitle
\section{Introduction}

We work over an algebraically closed field of characteristic 0.

Let $X \subset \p^{n+1}$ be a smooth hypersurface. We say that a vector bundle $E$ on $X$ is split if it is a direct sum of line bundles on $X$. We say that a vector bundle is $k$-generated if $k$ is the smallest integer such that there exists a surjection $\oplus_{i=1}^k \cO_X(a_i) \onto E$. A convenient notation for a coherent sheaf $\cF$ on $X$ is:
$$H^i_*(X, \cF) := \bigoplus_{m \in \Z} H^i(X, \cF(m))$$ 

We say a bundle is \textit{split} if it is a direct sum of line bundles. An \textit{arithmetically Cohen-Macaulay} (ACM) bundle on $X$ is a vector bundle $E$ satisfying $$H^i(X, E(m)) = 0, \,\, \forall \, m \in \Z \,\,\text{and} \,\,\, 0 < i < \text{dim}\, X$$

This definition is equivalent to saying that $\Gamma_*(X,E)$ is a maximal Cohen-Macaulay as $S_X$-module where $S_X$ is the graded ring corresponding to $X$.

A split bundle on a hypersurface or a projective space is obviously an ACM bundle. A theorem of Horrocks \cite{Hor} tells that over projective spaces, the converse is also true - A vector bundle on a  projective space splits if and only if it is an arithmetically Cohen-Macaulay bundle. One can ask if Horrock's criterion is true for arithmetically Cohen-Macaulay bundles over hypersurfaces?

The answer is no as there are examples of indecomposable arithmetically Cohen-Macaulay bundles over hypersurfaces (see \cite{M-R-R} or \cite{M-R-R3}). Though there is a conjecture in this direction:
\begin{proof}[ Conjecture (Buchweitz, Greuel and Schreyer \cite{BGS}):] Let $X \subset \p^n$ be a hypersurface. Let $E$ be an ACM bundle on $X$. If rank $E < 2^e$, where $e = \left [ \displaystyle{\frac{n-2}{2}} \right ]$, then $E$ splits. (Here $[q]$ denotes the largest integer $\leq q$.) 
\end{proof} 

This conjecture can not be strengthened further, as there exists  indecomposable arithmetically Cohen-Macaulay bundles of rank $2^e$ on such hypersurfaces for all degrees (see the construction in \cite{BGS}).

For low degree cases, we refer the reader to \cite{Kno} for $d = 2$ and to \cite{C-H} for the case of $d = 3$ surfaces in $\p^3$. 
For rank 2 on a general hypersurface of low degree in $\p^4$ and $\p^5$ we suggest \cite{C-M1}, \cite{C-M2}, \cite{C-M3}. 

For general hypersurfaces, Sawada \cite{S} found a sufficient condition for a arithmetically Cohen-Macaulay bundle to split depending upon the dimension as well as the degree of the hypersurface and rank of the vector bundle. His method uses matrix factorization (see Eisenbud \cite{Eisen2} for a background).

Rank two case is understood fairly well. A rank 2 arithmetically Cohen-Macaulay bundle on a hypersurface $X$ splits if:

\begin{enumerate}
\item dim$(X) \geq 5$ (see \cite{Kle} and \cite{M-R-R}).
\item dim$(X) = 4$ and $X$ is general hypersurface and $d \geq 3$ (see \cite{M-R-R} and \cite{R}). 
\item dim$(X) = 3$ and $X$ is general hypersurface and $d \geq 6$ (see \cite{M-R-R2} and \cite{R}).
\end{enumerate}

In a previous work \cite{Trip}, we showed that any rank 3 (resp. rank 4) ACM bundle over a hypersurface of dim$(X) \geq 7$ (resp. dim$(X) \geq 9$) is split. Our method employed a cohomological chase over a Koszul complex. In this note, we find an improvement for six generated ACM bundles,
\begin{theorem} \label{result_main} Let $E$ be any ACM bundle on a smooth hypersurface $X \subset \p^{n+1}$ where $n \geq 6$. If $E$ is six generated then $E$ is a split bundle.
\end{theorem}

\section{Preliminaries} \label{sec_prel}

Let $X \subset \p^{n+1}$ be a hypersurface of degree $d \geq 2$. Let $E$ be a rank $r$ ACM bundle on $X$ which is $k$-generated. The following facts are well known and will be used several times. The proofs for the same can be found (for instance) in \cite{Trip} (section 2).
\begin{itemize}
\item There exist split bundles $\wtil{F_1}, \wtil{F_0}$ of rank $k$ on $\p^{n+1}$ and a minimal resolution of $E$ on $\p^{n+1}$,
\begin{eqnarray*} 
0 \rar \wtil{F_1} {\rar} \wtil{F_0} \rar E \rar 0
\end{eqnarray*}
\item Restricting the above resolution to $X$ gives,
\begin{eqnarray}  \label{eqn_E_1_step}
0 \rar G \rar F_0 \rar E \rar 0
\end{eqnarray}
\begin{eqnarray}  \label{eqn_G_E}
0 \rar E(-d) \rar F_1 \rar G \rar 0
\end{eqnarray}

where $F_1, F_0$ are split bundles over $X$ of rank $k$ and $G$ is an arithmetically Cohen-Macaulay bundle.
\end{itemize}

Let $k$ be a positive integer. For any sequence $0 \rar \cF_0 \rar \cF_1 \rar \cF_2 \rar 0$ of vector bundles on a projective variety $Z$, there exists a resolution of $k$-th exterior power of $\cF_2$, 
$$0 \rar Sym^k(\cF_0) \rar \cdots \rar Sym^{k-i}(\cF_0) \otimes \wedge^i \cF_1 \rar \cdots \rar \wedge^k \cF_1 \rar \wedge^k \cF_2 \rar 0$$

In \cite{Trip} we have called this the \textit{$Sym-\wedge$ sequence of index $k$} \footnote[1]{We were unable to find any standard terminology in the literature for the given resolution.} associated to the given short exact sequence. Similarly there exists a resolution of $k$-th symmetric power of $\cF_2$ which we denoted as \textit{$\wedge-Sym$ sequence of index $k$} associated to the given sequence. 
$$0 \rar \wedge^k(\cF_0) \rar \cdots \rar \wedge^{k-i}(\cF_0) \otimes Sym^i \cF_1 \rar \cdots Sym^k \cF_1 \rar Sym^k \cF_2 \rar 0$$
For details we refer to \cite{Chris} and references provided therein.

Following lemma is from \cite{Trip}:
\begin{lemma} \label{G_direct_summand} Let $E$ be any bundle (not necessarily ACM) on a hypersurface $X \subset \p^{n+1}, \, n \geq 3$. Assume further that $H^1_*(X, E^{\vee}) = 0$. Let the exact sequence $0 \rar G \rar F_0 \rar E \rar 0$ be a minimal (1-step) resolution of $E$ on $X$. If $G$ admits a line bundle as a direct summand, then $E$ is split.
\end{lemma}

\section{Proof of the theorem}

\begin{proof}[Proof of theorem \ref{result_main}]: Let $E$ be a six generated arithmetically Cohen-Macaulay bundle on $X$ where dim$(X) \geq 6$. We take a minimal (1-step) resolution of $E$ on $X$: $$0 \lrar G \rar F_0 \rar E \rar 0$$ By assumption rk$(F_0) = 6$ therefore rk$(E) \leq 6$. The following cases are easily resolved:
\begin{enumerate}
\item If rk$(E) = 6$ then $E \cong F_0$ is split.
\item  If rk$(E) = 5$ then $G$ is a line bundle whence by lemma \ref{G_direct_summand} $E$ splits. 
\item If rk$(E) = 4$ then $G$ splits as it is then an ACM bundle of rank 2 on $X$ (see \cite{M-R-R}) whence $E$ splits (again by lemma \ref{G_direct_summand}).
\item If rk$(E) = 2$ then splitting is by results from \cite{M-R-R}.
\end{enumerate}

This leaves the case rk$(E) = 3$ open. $G$ (an ACM bundle) is also of rank 3 for this case. We have a minimal resolution of $E$ on $\p^{n+1}$: $$0 \rar \wtil{F_1} {\rar} \wtil{F_0} \rar E \rar 0$$ 
Taking exterior product, we get $$0 \rar \wedge^3 \wtil{F_1} {\rar} \wedge^3 \wtil{F_0} \rar \cF \rar 0$$
where $\cF$ is a coherent sheaf with support on $X$. It can be verified that $\cF$ is arithmetically Cohen-Macaulay sheaf which means that it is (infact) an ACM vector bundle on $X$ as $X$ is smooth. Restricting the above sequence to $X$ gives: 
\begin{eqnarray} \label{main_eqn}
0 \rar Tor^1(\cO_X, \cF) \rar \wedge^3 {F_1} {\rar} \wedge^3 {F_0} \rar \cF \rar 0
\end{eqnarray}

To compute the \textit{Tor} term, we tensor the following sequence by $\cF$:
$$0 \rar \cO_{\p ^{n+1}}(-d) \rar \cO_{\p ^{n+1}} \rar \cO_X \rar 0$$
to get $Tor^1(\cO_X, \cF) = \cF(-d)$. Thus we get the sequence: $$0 \rar \cF(-d) \rar \wedge^3 {F_1} {\rar} \wedge^3 {F_0} \rar \cF \rar 0$$

The map $F_1 \rar F_0$ factors via $G$ (see section \ref{sec_prel}). By functoriality of exterior product, the map $\wedge^3 {F_1} {\rar} \wedge^3 {F_0}$ will factor via $\wedge^3 G$. Thus the sequence above breaks up into 2 short exact sequences:
\begin{eqnarray} \label{eqn_F_G}
0 \rar \cF(-d) \rar \wedge^3 {F_1} \rar \wedge^3 G \rar 0
\end{eqnarray}\begin{eqnarray} \label{eqn_G_F}
0 \rar \wedge^3 G \rar \wedge^3 {F_0} \rar \cF \rar 0
\end{eqnarray}

Above sequences along with the fact that $G$ is rank 3 and $F_1, F_0$ are split bundles imply that $\cF$ is a split bundle - for example by verifying that $H^1_*(\cF) = 0$ which implies that equation \eqref{eqn_F_G} splits. 

Let $\cF_1 = \text{kernel}(\wedge^3 F_0 \onto \wedge^3 E)$. We have the following pullback diagram:
\begin{equation} \label{dgm_filtration_F_0}
\begin{gathered} \xymatrix{0 \ar[r]& \wedge^3 G \ar[r] & \cF_1 \ar[r] \ar@{^{(}->}[d]& \cE \ar[r] \ar@{^{(}->}[d]& 0 \\ 0 \ar[r] & \wedge^3 G \ar@{=}[u]  \ar[r] & \wedge^3 F_0 \ar[r] \ar[d] & \cF \ar[r] \ar[d] & 0& \\ & & \wedge^3 E \ar@{=}[r] \ar[d] & \wedge^3 E \ar[d]& \\ & & 0 & 0
}\end{gathered}
\end{equation}

Here the map $\wedge^3 G \into \cF_1$ is coming from the filtration diagram for $\wedge^3 F_0$ as induced by the sequence $0 \rar G \rar F_0 \rar E \rar 0$ and $\cE = \text{coker}(\wedge^3 G \rar \cF_1)$.

The above diagram (and the fact that $\cF$ is split) will imply that $\cE$ is split which in turn means that $\cF_1$ is a split bundle. Lemma \ref{lemma_final} (below) will now imply that $E$ is split.
\end{proof}

We complete the proof of the above theorem with following lemma which uses similar form of cohomological chase as done in \cite{Trip}:
\begin{lemma} \label{lemma_final} Let $E$ be an ACM vector bundle of rank 3 on a hypersurface $X$ of dimension $\geq 6$. Let $G, F_0$ denote the vector bundles on $X$ coming from the 1-step resolution of $E$ as in sequence \eqref{eqn_E_1_step}. Let $\cF_1 = \text{kernel}(\wedge^3 F_0 \onto \wedge^3 E)$ (as assumed in the proof of theorem \ref{result_main}). If $\cF_1$ is split then $E$ is split.
\end{lemma}
\begin{proof} Consider the following $\wedge - Sym$ sequence for index 3 associated with sequence \eqref{eqn_G_E}:

$
0 \rar \wedge^3 E(-d) \rar \wedge^2 E(-d) \otimes F_1 \rar E(-d) \otimes Sym^{2} F_1 \rar Sym^3 F_1 \rar Sym^3 G \rar 0$

Breaking it into short exact sequences and using the fact that as $E$ is rank 3 therefore $\wedge^3 E(-d), \,\wedge^2 E(-d), \,E(-d)$ are all arithmetically Cohen-Macaulay bundle and $Sym^i F_1$ is a split bundle for all $i$, we get $$H^i_*(Sym^3 G) = 0 \,\,\,\text{for}\,\,\, i = 1,2 \ldots ,\text{dim}(X) - 4$$ Similarly we write the $\wedge - Sym$ sequence for index 2: $$0 \rar \wedge^2 E(-d) \rar E(-d) \otimes F_1 \rar Sym^{2} F_1 \rar Sym^2 G \rar 0$$ which gives $H^i_*(Sym^2 G) = 0$ for $i = 1,2 \ldots ,\text{dim}(X) - 3$.

Now we write the $Sym - \wedge$ sequence for index 3 for the short exact sequence \eqref{eqn_E_1_step}:$$0 \rar Sym^3 G \rar Sym^2 G \otimes F_0 \rar G \otimes \wedge^2 F_0 \rar\wedge^3 F_0 \rar \wedge^3 E \rar 0$$ 

Breaking it up we get $0 \rar Sym^3 G \rar Sym^2 G \otimes F_0 \rar J_1 \rar 0$. The vanishing results for cohomologies of $Sym^3 G, \, Sym^2 G$ implies that if dim$(X) \geq 6$ then $H^1_*(X, J_1) = 0$. Now $J_1$ further fits into the following short exact sequence: \begin{eqnarray}
\label{eqn_G_split}
0 \rar J_1 \rar G \otimes \wedge^2 F_1 \rar \cF_1 \rar 0 \end{eqnarray} where $\cF_1$ is split by assumption. Therefore \eqref{eqn_G_split} is a split sequence and hence $G$ admits a line bundle as a direct summand. Lemma \ref{G_direct_summand} tells us that $E$ is split. 
\end{proof}


\end{document}